\documentclass[10pt, reqno]{amsart}
\usepackage[T1]{fontenc}
\usepackage[utf8]{inputenc}
\usepackage{amsmath,amsfonts,amsthm,amssymb,amsxtra}
\usepackage{mathrsfs}  
\usepackage[colorinlistoftodos,prependcaption]{todonotes}

\usepackage{hyperref}
\usepackage{bbm} 
\usepackage{color}
\usepackage[margin=3cm]{geometry}
\allowdisplaybreaks
\usepackage{enumitem}
\setlist{  
listparindent=\parindent,
parsep=0pt,
}
\usepackage{soul}

\newtheorem{theorem}{Theorem}[section]

\newtheorem{lemma}[theorem]{Lemma}
\newtheorem{corollary}[theorem]{Corollary}

\theoremstyle{definition}

\theoremstyle{remark}
\newtheorem{remark}[theorem]{Remark}

\numberwithin{equation}{section}

\DeclareMathOperator{\supp}{supp}

\renewcommand{\epsilon}{t}

\newcommand{\Ds}{{(-\Delta)^{s}}}

\renewcommand{\d}{\mathrm{d}}
\newcommand{\dd}{\, \mathrm{d}}
\newcommand{\R}{\ensuremath{\mathbb R}} 

\newcommand{\pv}{\, \mathrm{p.v.}\,}

\begin{document}

\title[Fractional Hardy-Rellich inequalities via integration by parts]{Fractional Hardy-Rellich inequalities \\ via integration by parts }

\author[N. De Nitti]{Nicola De Nitti}
\address[N. De Nitti]{Friedrich-Alexander-Universität Erlangen-Nürnberg, Department of Mathematics, Chair for Dynamics, Control, Machine Learning and Numerics (Alexander von Humboldt Professorship), Cauerstr. 11, 91058 Erlangen, Germany.}
\email{nicola.de.nitti@fau.de}

\author[S. M. Djitte]{Sidy Moctar Djitte}
\address[S. M. Djitte]{Friedrich-Alexander-Universität Erlangen-Nürnberg, Department of Mathematics, Chair for Dynamics, Control, Machine Learning and Numerics (Alexander von Humboldt Professorship), Cauerstr. 11, 91058 Erlangen, Germany. 
}
\email{sidy.m.djitte@fau.de}


\begin{abstract}
We prove a fractional Hardy-Rellich inequality with an explicit constant in bounded domains of class $C^{1,1}$. The strategy of the proof generalizes an approach pioneered by E. Mitidieri (\emph{Mat. Zametki}, 2000) by relying on a Pohozaev-type identity.
\end{abstract}

\maketitle

\section{Introduction}\label{sec:intro}
In \cite{MR1544414}, G. H. Hardy proved that, if $p > 1$ and $f$ is a non-negative function in $L^p(\R_+)$, then $f$ is integrable over the interval $(0, x)$ for every $x>0$ and 
\begin{align*}
\int_0^\infty\left(\frac{1}{x} \int_0^x f(t) \dd t \right)^p \dd x \le \left(\frac{p}{p-1} \right)^p \int_0^\infty f(x)^p \dd x
\end{align*}
holds; or, letting  $u(x) = \int_0^x f(t) \dd t$, 
\begin{align*}
    \int_0^\infty \frac{u(x)^p}{x^p} \dd x \le \left( \frac{p}{p-1}\right)^p \int_0^\infty |u'(x)|^p \dd x.
\end{align*}
The constant $(p/(p-1))^p$ was proved to be sharp by Landau in \cite{zbMATH02586070}. 
From these beginnings, many Hardy-type inequalities have been proven and have become fundamental tools in several branches of analysis. For further information and historical context, we refer to the surveys \cite{zbMATH07213519, MR1069756, MR2256532}.

The classical $N$-dimensional generalization of the Hardy inequality states that, for $N>1$, $1 \leq p<\infty$, with $p \neq N$, and for all $u \in C^{1,1}_c(\mathbb{R}^N \setminus\{0\})$,
$$
\int_{\mathbb{R}^N} \frac{|u(x)|^p}{|x|^p} \dd x \leqslant\left(\frac{p}{|N-p|}\right)^p \int_{\mathbb{R}^N} |\nabla u(x)|^p \dd x.
$$
More precisely, $u$ may belong to $W^{1, p}\left(\mathbb{R}^N\right)$ when $1 \leq p<N$ and $W^{1, p}\left(\mathbb{R}^N \setminus\{0\}\right)$ when $N<p<\infty$. 
Here the constant  $(p/|N-p|)^p$ is sharp and is not attained in these  Sobolev spaces. If $p=1$, equality holds for any positive symmetric decreasing function. For these and additional results, including discussions on the optimal value of the constant involved, we direct readers to, without claiming comprehensiveness, \cite{MR2048514, MR2048513, MR1431208, MR1769903, MR3408787} and the references cited therein.

A related inequality is due to Rellich \cite{MR0240668}:
$$
\frac{N^2(N-4)^2}{16} \int_{\mathbb{R}^N} \frac{|u({x})|^2}{|{x}|^4} \dd x \leq \int_{\mathbb{R}^N}|\Delta u({x})|^2 \dd x, 
$$
for $u \in W^{2,2}(\mathbb{R}^N)$ and $N>4$. We point to \cite{MR3408787,MR4480576} and the references therein for further information on Rellich-type inequalities (in particular, e.g., \cite[Section 6.4.2]{MR3408787} for comments on the cases $N \in \{2,3,4\}$).

In the present contribution, we shall focus on the following $L^p$ Rellich inequality (see, e.g., \cite{MR1769903}):  
\begin{align}\label{eq:hoh}
    c_{p, \theta}^p \int_{\Omega} \frac{|u(x)|^p}{|x|^{\theta+2}} \dd x \leq \int_{\Omega} \frac{|\Delta u(x)|^p}{|x|^{\theta+2-2 p}} \dd x,
\end{align}
for $u \in C^{1,1}_c(\Omega)$, $p>1$, $N>\theta+2$, $\theta \in \mathbb{R}$, $\Omega \subset \R^N$ a smooth bounded domain, and sharp constant given by  
$
c_{p, \theta}=(N-2-\theta)[(p-1)(N-2)+\theta]/p^2.
$
Our main aim is to establish a counterpart of \eqref{eq:hoh} where the Laplace operator is replaced by the \emph{fractional Laplacian operator}: 
\begin{align} \label{eq:hardy-intro}
\int_{\Omega} \frac{|u|^{p}(x)}{|x|^{\theta+2 s}} \dd x \lesssim  \int_{\Omega} \frac{\left|(-\Delta)^{s} u(x)\right|^{p}}{|x|^{\theta+2 s-2sp}} \dd x.
\end{align}

Several Hardy-type inequalities are already available for fractional operators. For instance, in \cite{MR436854}, Herbst proved that 
\begin{align}\label{eq:herbst}
 \tilde{\mathcal{C}}^p_{N, s, p} \int_{\mathbb{R}^N} \frac{|u(x)|^p}{|x|^{p s}} \dd x \leq \|(-\Delta)^{s / 2} u\|_{L^p(\R^N)}^p 
 \end{align}
for $1<p<\infty$, $s>0$, $N>p s$ and $u \in C^{1,1}_c\left(\mathbb{R}^N\right)$, 
with the optimal constant 
$
\tilde{\mathcal{C}}_{N, s, p}=2^{-s} \frac{\Gamma\left(\frac{N(p-1)}{2 p}\right) \Gamma\left(\frac{N-p s}{2 p}\right)}{\Gamma\left(\frac{N}{2 p}\right) \Gamma\left(\frac{N(p-1)+p s}{2 p}\right)} .
$
A related result is contained in \cite[Theorem 1.1]{MR2469027} (see  also  \cite{MR1986163}),  where it was shown that 
\begin{align}\label{eq:h}
\mathcal{C}_{N, s, p} \int_{\mathbb{R}^N} \frac{|u(x)|^p}{|x|^{p s}} \dd x \leq \iint_{\mathbb{R}^N \times \mathbb{R}^N} \frac{|u(x)-u(y)|^p}{|x-y|^{N+p s}}\dd x \dd y,
\end{align}
for all $u\in W^{s,p}(\R^N)$ with optimal constant\footnote{ See also \cite[Eq (3.5) and Eq. (3.6)]{MR2469027} for equivalent expressions of $\mathcal C_{N,s,p}$. Moreover, from \cite[Eq. (1.6)]{MR2469027}, if $p=2$, then the constant is given more explicitly by 
$$
\mathcal{C}_{N, s, 2}=2 \pi^{N / 2} \frac{\Gamma((N+2 s) / 4)^2}{\Gamma((N-2 s) / 4)^2} \frac{|\Gamma(-s)|}{\Gamma((N+2 s) / 2)}.$$} 
\begin{align} \label{eq:c-sharp}
\mathcal{C}_{N, s, p}:=2 \int_0^1 r^{p s-1}\left|1-r^{(N-p s) / p}\right|^p \Phi_{N, s,  p}(r) \dd r,
\end{align} 
where
\begin{align*}
\Phi_{N, s, p}(r)&:=\mathrm{vol}(\mathbb{S}^{N-2}) \int_{-1}^1 \frac{\left(1-t^2\right)^{\frac{N-3}{2}}}{\left(1-2 r t+r^2\right)^{\frac{N+p s}{2}}}  \dd t, && N \ge 2, \\
\Phi_{1, s, p}(r)&:=\left(\frac{1}{(1-r)^{1+p s}}+\frac{1}{(1+r)^{1+p s}}\right), &&N=1 .
\end{align*}
For $p=2$, the right-hand side of \eqref{eq:herbst}  is proportional to the one in \eqref{eq:h} because of \cite[Proposition 3.6]{MR2944369}.
We refer also to \cite{MR2425175, MR1717839,MR2659764,MR4480576,zbMATH07319433,MR3237044,MR4062961,MR3814363,2209.03012,zbMATH07213519} for further results. 
In particular, higher-order versions of the fractional Hardy inequality are contained in \cite{MR4160667,zbMATH06446451} and a survey of some fractional counterparts of the Rellich inequality are also available in the recent monographs \cite{MR3408787,MR4480576}. 

However, to the best of our knowledge, no fractional analogue of \eqref{eq:hoh} of the type \eqref{eq:hardy-intro} is available in the literature. The main aim of this paper is to prove it (see Theorem \ref{th:main} below for the precise statement) by generalizing the strategy pioneered by Mitidieri in \cite[Section 3]{MR1769903}, where the key observation is that \eqref{eq:hoh} can be deduced from an identity of Rellich-Pohozaev type by suitably choosing an auxiliary function. This approach does not seem to yield sharp constants in the fractional context. On the other hand, we believe that it is of interest in itself for its simplicity. The key ingredient of the proof is a fractional Pohozaev-type identity (see \cite{2112.10653,MR3211861}).

\subsection{Mitidieri's approach to the Hardy-Rellich inequality}
\label{ssec:mit}
For the sake of completeness, let us  outline the strategy employed in \cite[Section 3]{MR1769903} to prove that, for any function $u \in C^{1,1}_c(\Omega)$, the Hardy-Rellich inequality  \eqref{eq:hoh} holds. The starting point is the following \emph{Rellich-Pohozaev identity} (see \cite[Corollary 2.1]{MR1211727}): for all $u, v \in C^2(\bar{\Omega})$, we have
\begin{align}\label{eq:rp-m}
\begin{split}
    &\int_{\Omega}\Big((x \cdot \nabla v)\Delta u + (x \cdot \nabla u) \Delta v\Big)  \dd x \\ &\qquad =(N-2) \int_{\Omega}\nabla u \cdot \nabla v \dd x +\int_{\partial \Omega}\Big(\partial_\nu u (x \cdot  \nabla  v)+\partial_\nu v (x\cdot \nabla  u) - (\nabla u \cdot  \nabla  v) (x \cdot \nu) \Big) \dd \sigma,
\end{split}
\end{align}
where $\nu$ denotes the outward-pointing normal unit vector and $\partial_{\nu}$ denotes the external normal derivative at the point $x \in \partial \Omega$. In particular, we shall consider $u \in C^2_c(\Omega)$, which cancels out the boundary terms. The main point of the argument is the choice of the auxiliary functions in \eqref{eq:rp-m} and it goes as follow. Assuming $u>0$ (in its support) and $p>1$, we plug $u \mapsto u^p$ and $v \mapsto v_t:=(|x|^\theta + t)^{-1}$ (with $t >0$ and $\theta \in \R$) into \eqref{eq:rp-m}. Letting $t \to 0^+$ and applying  the chain rule $
\Delta \phi(u) = \phi'(u)\Delta u + \phi''(u) |\nabla u|^2
$ (with $\phi(u) = u^p$) yields 
\begin{align}\label{eq:m1}
\begin{split}
&\int_{\Omega} \frac{u^{p-1} \Delta u }{|x|^\theta} \dd x+(p-1) \int_{\Omega} \frac{u^{p-2}|\nabla u|^2}{|x|^{\theta}} \dd x =\theta\int_{\Omega} \frac{  u^{p-1}x \cdot \nabla u}{|x|^{\theta+2}} \dd x.
\end{split}
\end{align}
Using the divergence theorem the RHS of \eqref{eq:m1} can be written as 
\begin{align}\label{eq:m2}
\begin{split}
\int_{\Omega} \frac{u^{p-1}x \cdot \nabla u}{|x|^{\theta+2}} \dd x & = -\frac{N-2-\theta}{p} \int_{\Omega} \frac{u^p}{|x|^{\theta+2}} \dd x.
\end{split}
\end{align} 
On the other hand, by Cauchy-Schwarz and H\"older's inequalities, we have
\begin{align*}
 \frac{N-2-\theta}{p} \int_{\Omega} \frac{u^p}{|x|^{\theta+2}} \dd x = - \int_{\Omega} \frac{u^{p-1}x \cdot \nabla u}{|x|^{\theta+2}} \dd x 
&\le \int_{\Omega} \frac{|\nabla u| \, |x| \, |u|^{p-1}}{|x|^{\theta+2}} \dd x  
\\&\le \left(\int_{\Omega} \frac{|\nabla u|^2|u|^{p-2}}{|x|^{\theta}} \dd x \right)^{\frac{1}{2}}\left(\int_\Omega \frac{u^p}{|x|^{\theta +2}}\dd x \right)^{\frac{1}{2}},
\end{align*}
from which we deduce
\begin{align}\label{eq:m3}
\int_{\Omega} \frac{|\nabla u|^2|u|^{p-2}}{|x|^{\theta}} \dd x \geq  \left(\frac{N-2-\theta}{p}\right)^2 \int_{\Omega} \frac{u^p}{|x|^{\theta+2}} \dd x.
\end{align} 
Finally, applying H\"older's inequality with conjugate exponents $p$ and $p/(p-1)$, we estimate 
\begin{align}\label{eq:hm}
\begin{split}
\int_{\Omega} \frac{u^{p-1} \Delta u }{|x|^\theta} \dd x &\le \int_{\Omega}  \frac{u^{p-1} }{|x|^\alpha} \, \frac{ |\Delta u| }{|x|^{\theta-\alpha}} \dd x 
\\  &\le \left(\int_{\Omega} \frac{u^p}{|x|^{\frac{\alpha p}{p-1}}} \dd x \right)^{\frac{p-1}{p}} \left(\int_{\Omega} \frac{|\Delta u|^p}{|x|^{(\theta-\alpha)p}} \dd x\right)^{\frac{1}{p}},
\end{split} 
\end{align}
for $\alpha \in \R$. In particular, taking $\alpha = (\theta + 2){(p-1)}/{p}$ and plugging \eqref{eq:m2}, \eqref{eq:m3}, and \eqref{eq:hm} into \eqref{eq:m1} yield
$$
\left[\frac{N-2-\theta}{p^2}\Big((p-1)(N-2)+\theta\Big)\right] \int_{\Omega} \frac{u^p}{|x|^{\theta+2}} \dd x \leq\left(\int_{\Omega} \frac{u^p}{|x|^{\theta+2}} \dd x\right)^{\frac{p-1}{p}}\left(\int_{\Omega} \frac{|\Delta u|^p}{|x|^{\theta+2-p}} \dd x\right)^{\frac{1}{p}}.
$$
That is, 
$$
c_{p, \theta}^p \int_{\Omega} \frac{u^p}{|x|^{\theta+2}} \dd x \leq \int_{\Omega} \frac{|\Delta u|^p}{|x|^{\theta+2-2 p}} \dd x,
$$
with 
$$
c_{p, \theta}=\frac{(N-2-\theta)[(p-1)(N-2)+\theta]}{p^2}.
$$
To remove the extra assumption $u>0$, we use  $u_t:=(u^2+t^2)^{{1}/{2}}-t$ (with $t >0$) in the computations above and finally let $t \to 0^+$.

\subsection{Extension to $s \in (0,1)$ and outline of the paper} 

To extend the argument outlined in Section \ref{ssec:mit} to the fractional setting, several technical difficulties ensue. 

Our starting point is the generalized fractional Pohozaev identity proved in \cite[Theorem 1.3]{2112.10653}.  However, functions of the form $v_t(\cdot):=(t^2+|\cdot|^2)^{-\theta/2}$ are not admissible in the identity \cite[Theorem 1.3]{2112.10653}. To overcome this issue, in the first step, we use an approximation technique and replace  $v_t$ by $z_kv_t$, where $z_k$ is a suitable cut-off function supported in $\Omega$ (see Lemma \ref{lm:dist}). We expand the resulting identity by using the product rule for the fractional Laplacian and then pass to the limit to obtain the fractional version of \eqref{eq:m1}. Along the way, the expression $\lim _{t \rightarrow 0^{+}}(-\Delta)^{s}((t^2+|\cdot|^2)^{-\theta/2})(x)$ has to be computed explicitly. This computation might be well-known, but we present it here for the sake of completeness (see Lemma \ref{lm:compu}). Finally, since the classical chain rule 
$
\Delta \phi(u) = \phi'(u)\Delta u + \phi''(u) |\nabla u|^2
$
does not hold for the fractional Laplacian (we only have \cite[Lemma 2.6]{MR4288669} instead), we need to estimate $(-\Delta)^s u^p$ using Cordoba-Cordoba's inequality (see \cite[Theorem 1]{MR2032097} and \cite[Theorem 1.1]{MR3642734}) instead. 
In other words, in the fractional computation, the bound corresponding to \eqref{eq:m3} is replaced by a much rougher one. 

The paper is organized as follows. In Section \ref{sec:main}, we state our main theorems and present the needed preliminary notions.  The proofs are developed in Section \ref{sec:proofs}. Some technical lemmas used in the arguments are collected in Appendix \ref{app:prelim}.

\section{Main  results}\label{sec:main}
Let $0<s<1$. The fractional Laplacian operator $\Ds$ is defined, when acting on sufficiently regular functions, say $u \in C^{1,1}_c(\R^N)$, through the singular integral   
\begin{equation}
\label{frac-lap-sg-int}
(-\Delta)^s u(x) := c_{N,s} \,  \mathrm{p.v.}  \int_{\R^N} \frac{u(x) - u(y)}{|x-y|^{N+2s}} \dd y,
\end{equation}
where $\mathrm{p.v.}$ stands for the Cauchy principal value and the normalising constant $c_{N,s}$ is given by 
\begin{align}\label{cns}
c_{N,s} :=   \frac{s 2^{2s} \Gamma(\frac{N+2s}{2})}{\pi^{N/2} \Gamma (1-s)}
\end{align}
(see \cite{MR2944369}).  It can also be defined weakly for any $u\in H^s(\R^N)$ by letting
$$
\big<(-\Delta)^su,v\big>:=\frac{c_{N,s}}{2}\iint_{\R^N\times\R^N}\frac{(u(x)-u(y))(v(x)-v(y))}{|x-y|^{N+2s}}\dd x\dd y=:\mathcal{E}_s(u,v)
$$
for all $v\in H^s(\R^N)$. Here, $H^s(\R^N)$ is the subspace of $L^2$-functions $u$ for which the seminorm $\mathcal{E}_s(u,u)$ is finite. 

Throughout this manuscript, we denote by $\mathcal{H}^s_0(\Omega)$ the subset of $L^2$-functions belonging to $H^s(\R^N)$ and such that $u\equiv 0$ in $\R^N\setminus\Omega$. We recall that, by  \cite[Theorem 1.4.2.1]{MR775683}, if $\Omega$ has a continuous boundary, then $\mathcal{H}^s_0(\Omega)$ coincides with the closure of $C^\infty_0(\Omega)$ with respect to the seminorm $\mathcal{E}_s(\cdot,\cdot)$.

Our first main result is the following integral identity. Roughly speaking, this corresponds to a generalized nonlocal version of \eqref{eq:m1}. 

\begin{theorem}[Rellich-Pohozaev-type identity with Hardy weights]\label{main-identityt}
Let  $\theta\geq 0$ and $N\in\mathbb{N}$ with $N> \theta+2s$. Let $\Omega \subset \R^N$ be a bounded open set of class $C^{1,1}$ containing the origin $0$. Let $Y \in C^{0,1}\left(\mathbb{R}^{N}, \mathbb{R}^{N}\right)$ with $Y(0)=0$. We define the nonlocal operator 
\begin{equation}\label{eq:op-def-kernel}
    \left[\mathscr{L}_{\mathcal{K}_{Y}} u\right](x):= \mathrm{p.v.}  \int_{\mathbb{R}^{N}}(u(x)-u(y)) \mathcal{K}_{Y}(x, y) \dd y,
\end{equation}
where
\begin{equation}\label{eq:def-kernel}
\mathcal{K}_{Y}(x, y)=\frac{c_{N, s}}{2}\left[\operatorname{div} Y(x)+\operatorname{div} Y(y)-(N+2 s) \frac{(Y(x)-Y(y)) \cdot(x-y)}{|x-y|^{2}}\right]|x-y|^{-N-2 s} .
\end{equation}
Let $u \in C^{1,1}_c(\Omega)$ and define $U_{t,p}= (u^2+t^2)^{p/2}-t^p\in C^{1,1}_c(\Omega)$ for all $t> 0$. Then for all $p$, the following identity holds
\begin{align}\label{eq:17}
\begin{split}
&-b_{N, s, \theta} \int_{\Omega} \frac{U_{t,p}(x)}{|x|^{\theta+2 s}} \operatorname{div} Y(x) \dd x+b_{N, s, \theta}(\theta+2 s) \int_{\Omega} U_{t,p}(x) \frac{x \cdot Y(x)}{|x|^{\theta+2 s+2}} \dd x \\ &=\theta \int_{\Omega} \frac{x \cdot Y(x)}{|x|^{\theta+2 }}(-\Delta)^{s} U_{t,p}\dd x-\int_{\Omega} \frac{\mathscr{L}_{\mathcal{K}_{Y}} U_{t,p}}{|x|^{\theta}} \dd x,
\end{split}
\end{align}
 where 
\begin{align*}
b_{N, s, \theta}&:=c_{N,s}\int_{0}^1r^{2s-1}(1-r^\theta)(1-r^{N-2s-\theta})\psi_N(r)\dd r,
  \end{align*}   
 with $c_{N,s}$ given as in \eqref{cns} and 
\begin{align*} 
 \psi_N(r)&:=2\,\mathrm{vol}(\mathbb{S}^{N-2})\int_{-1}^1\frac{(1-h^2)^{\frac{N-3}{2}}}{\big(1+r^2-2rh\big)^{(N+2s)/2}}\dd h && \text{if $N\geq 2$}, \\ 
 \psi_N(r)&:=2\left(\frac{1}{(1-r)^{1+2 s}}+\frac{1}{(1+r)^{1+2 s}}\right) && \text{if $N=1$}. 
\end{align*} 
\end{theorem} 
We note that all the quantities in \eqref{eq:17} are well defined by assumption. As a consequence of identity \eqref{eq:17}, we obtain the following fractional Hardy-Rellich inequalities.
\begin{theorem}[Generalized fractional Hardy-Rellich-type inequality in bounded domains]\label{th:main}
Let  $\theta\geq 0$ and $N\in\mathbb{N}$ with $N> \theta+2s$. Let $\Omega \subset \R^N$ be a bounded open set of class $C^{1,1}$ containing the origin $0$. Then, for all $u \in C^{1,1}_c(\Omega)$, we have 
\begin{align}\label{eq:h-main-1} 
\left[\frac{{b_{N, s, \theta}}}{p}\right]^{p}\int_{\Omega} \frac{|u|^{p}}{|x|^{\theta+2 s}} \dd x \leq \int_{\Omega} \frac{\left|(-\Delta)^{s} u\right|^{p}}{|x|^{\theta+2 s-2sp}} \dd x
\end{align}
for $p>1$ and 
\begin{align}\label{eq:h-main-2} 
b_{N, s, \theta}\int_{\Omega} \frac{|u|}{|x|^{\theta+ 2s}} \dd x \leq \int_{\Omega} \frac{\mathrm{sign}(u)(-\Delta)^{s} u}{|x|^{\theta}} \dd x
\end{align}
for $p=1$.
\end{theorem}
Taking $\theta=2sp-2 s$ in \eqref{eq:h-main-1} above yields the following bounded region version of a result by Herbst \cite{MR436854} (see also \cite[Theorem 1.3]{MR4062961} and \cite{ MR1717839} for related results).
\begin{corollary} (Improved fractional Hardy inequality)\label{cor:h}
Let $s\in(0,1)$, ${N}/{p}> 2s$, and $p>1$. Let $\Omega\subset \R^N$ be a bounded open set of class $C^{1,1}$ containing the origin $0$. Then, for all $u \in C^{1,1}_c(\Omega)$, we have
\begin{equation}\label{eq:cor. 2.3}
    \left[\frac{{b_{N, s, 2s(p-1)}}}{p}\right]^{p}\int_{\Omega}\frac{|u|^p}{|x|^{2sp}}\dd x\leq \int_{\Omega}|(-\Delta)^su|^p\dd x.
\end{equation}

In particular, for $p=2$ and $N> 4s$, we deduce
	\begin{align} \label{eq:h-rem-m}
	\begin{split} 
	&\left[\frac{b_{N, s / 2, s}}{2}\right]^{2}\int_{\Omega} \frac{u^{2}(x)}{|x|^{2 s}} \dd x \leq \int_\Omega |(-\Delta)^{s/2} u|^2 \dd x 
\\ & \qquad \leq \frac{c_{N,s}}{2}\Bigg[\iint_{\R^N\times\R^N} \frac{(u(x)-u(y))^{2}}{|x-y|^{N+2 s}} \dd x\dd y - \int_{\mathbb{R}^{N} \setminus \Omega}\left(\int_{\Omega} \frac{u(y)}{|x-y|^{N+s}} \dd y\right)^{2} \dd x\Bigg]
	\end{split} 
	\end{align} 
for all $u \in \mathcal{H}^s_0(\Omega)$. If $N=1$ and $0<s< \frac{1}{4}$, the estimate \eqref{eq:h-rem-m} reduces to 
 \begin{align} \label{eq:h-rem-m-1}
	\begin{split} 
	&\left[\frac{b_{1, s / 2, s}}{2}\right]^{2}\int_{-1}^{+1} \frac{u^{2}(x)}{|x|^{2 s}} \dd x 
\\ & \leq \frac{c_{1,s}}{2}\left[\iint_{\R^2} \frac{(u(x)-u(y))^{2}}{|x-y|^{1+2 s}} \dd x\dd y - \int_{1}^\infty\left(\int_{-1}^{+1} \frac{u(y)}{(x-y)^{1+s}} dy\right)^{2}\dd x- \int_{1}^\infty\left(\int_{-1}^{+1} \frac{u(y)}{(x+y)^{1+s}} dy\right)^{2}\dd x\right].
	\end{split} 
	\end{align} 
\end{corollary} 
Before we give the proofs, let us collect a few remarks on the main results. 
\begin{remark}[Notes on the main result]  \,
\begin{enumerate}
\item  When $\theta=0$, the constant $b_{N,s,0}$ vanishes and equation $(2.5)$ gives, in particular,  that 
\begin{equation}\label{eq:k-ob}
   \int_{\Omega}\operatorname{sign}(u)(-\Delta)^s u\dd x\geq 0,\quad\text{ for all $u\in C^{1,1}_c(\Omega)$.} 
\end{equation}
An estimate like \eqref{eq:k-ob} follows also by using a Kato-type inequality and the symmetry of the fractional Laplace operator  (see \cite{MR3865202}):
\begin{align*}
    \int_{\R^N} \operatorname{sign}(u) \, (-\Delta)^s u \dd x  \ge \int_{\R^N} (-\Delta)^s |u| \dd x = 0, \quad \text{for all $u\in C^{1,1}_c(\R^N)$.} 
\end{align*}
\item  If we replace $\Omega$ by $\mathbb R^N$, using the identity 
 $$
 \int_{\R^N}x\cdot\nabla u(-\Delta)^su\dd x=(2s-N)\int_{\R^N}u(-\Delta)^su\dd x\quad\text{for all $u\in C^{1,1}_c(\R^N)$,}
 $$
we prove by a similar argument (much simpler in fact)  that 
$$
\left[\frac{b_{N, s, \theta}}{p}\right]^{p}\int_{\mathbb{R}^{N}} \frac{|u|^{p}}{|x|^{\theta+2 s}} \dd x \leq \int_{\mathbb{R}^{N}}\frac{\left|(-\Delta)^{s} u\right|^{p}}{|x|^{\theta+2 s-2sp}} \dd x,
$$
for all $\theta > -2s$, \,$p\in(1,\infty)$, and for all $u \in C^{1,1}_c\left(\mathbb{R}^{N}\right)$.
\item   Taking the limit when $s \to 1^-$ in \eqref{eq:h-main-1}, we recover the inequality \eqref{eq:hoh} announced in \cite{MR1769903} with the constant $\bigg[\frac{2\theta}{p}\frac{\Gamma(\frac{N-\theta}{2})}{\Gamma(\frac{N-\theta-2}{2})}\bigg]^p$. Indeed, from the fact that
$$
\lim_{t\to 0^+}(-\Delta)^s(t^2+|\cdot|^2)^{-\frac{\theta}{2}}(x)=(-\Delta)^s|\cdot|^{-\theta}(x),\quad \text{ for all } x\in\R^N\setminus\{0\},
$$
which can be seen by computing $(-\Delta)^s(|\cdot|^{-\theta})(x)$ as in Lemma \ref{lm:compu})
and the identity
\begin{align}\label{eq:lap-xq}
    (-\Delta)^s|\cdot|^{-\theta}(x) = 2^{2s} \frac{\Gamma\left(\frac{N-\theta}{2}\right) \Gamma(\frac{2s+\theta}{2})}{\Gamma\left(\frac{N-\theta-2s}{2}\right) \Gamma\left(\tfrac{\theta}{2}\right)}  \  |x|^{-(\theta + 2s)}, \quad \text{ for all }  x \neq 0,\quad N>\theta > - 2s,
\end{align}
(contained in \cite[Table 1]{MR3888401}), we deduce that 
$$
b_{N,s,\theta}= 2^{2s} \frac{\Gamma\left(\frac{N-\theta}{2}\right) \Gamma(\frac{2s+\theta}{2})}{\Gamma\left(\frac{N-\theta-2s}{2}\right) \Gamma\left(\tfrac{\theta}{2}\right)}\longrightarrow \frac{2\theta\Gamma(\frac{N-\theta}{2})}{\Gamma\left(\frac{N-\theta-2}{2}\right)}\quad\text{as}\quad s\to 1^-.
$$
 
\item We believe that more general Hardy-type inequalities may be obtained by choosing a suitable vector field $Y$ in \eqref{eq:17}. Indeed, as we shall see, the estimate \eqref{eq:h-rem-m} follows from \eqref{eq:17} by taking $Y\equiv \textrm{id}_{\R^N}$.
\end{enumerate}
\end{remark}

\section{Proof of the main theorems}
\label{sec:proofs}

In this section, we prove our main results. We start with Theorem \ref{main-identityt}.
\begin{proof}[Proof of Theorem \ref{main-identityt}] We recall the following identity from \cite[Lemma 2.1]{2112.10653}. Let $\Omega$ be a bounded open set and let  $u\in C^{1,1}_c(\Omega)$. Let $Y:\R^N\to\R^N$ be a globally Lipschitz vector field. Then, denoting 
\begin{align}
    \mathcal{E}_Y(u,u):=\iint_{\R^N\times\R^N}(u(x)-u(y))^2\mathcal{K}_Y(x,y)\dd x \dd y,
\end{align}
where
\begin{align}\label{eq:defor-kernel}
   \mathcal{K}_Y(x,y)=\frac{c_{N, s}}{2}\left[\operatorname{div} Y(x)+\operatorname{div} Y(y)-(N+2 s) \frac{(Y(x)-Y(y)) \cdot(x-y)}{|x-y|^{2}}\right]|x-y|^{-N-2 s},
\end{align}
we have 
\begin{align}\label{eq:3.3}
    \mathcal{E}_Y(u,u)=-2\int_{\Omega}Y\cdot\nabla u(-\Delta)^su\dd x.
\end{align}
Consequently, we have
	\begin{align}\label{eq:p}
	\begin{split}
		\int_{\Omega} Y \cdot \nabla u(-\Delta)^{s} v \dd x+\int_{\Omega} Y \cdot \nabla v(-\Delta)^{s} u \dd x=\mathcal{E}_Y(u,v),\quad \text{for all $u,v\in C^{1,1}_c(\Omega)$.}
	\end{split}
	\end{align}
In \eqref{eq:p}, we replace $u$ by $U_{t,p}= (u^2+t^2)^{p/2}-t^p$ and $v$ by $v_{k}=\frac{1-\rho_{k}}{(t^2+|\cdot|^2)^\frac{\theta}{2}}:=\frac{z_k}{(t^2+|\cdot|^2)^\frac{\theta}{2}}$ with $t>0$ where $\rho_k$ is defined as in Lemma \ref{lm:dist}. By the same lemma, we know that  $v_{k}$ is admissible in \eqref{eq:p}. With this substitution, we deduce  
	\begin{align}\label{eq:32}
	\begin{split}
	&\underbrace{\int_{\Omega} Y \cdot \nabla U_{t,p}(-\Delta)^{s}\left(\frac{z_k}{(t^2+|x|^2)^\frac{\theta}{2}}\right) \dd x}_{=:I_k}+\int_{\Omega} z_k Y \cdot \nabla\left(\frac{1}{(t^2+|x|^2)^{\frac{\theta}{2}}}\right)(-\Delta)^{s} U_{t,p} \dd x \\ &\qquad +\underbrace{\int_{\Omega} \frac{Y \cdot \nabla z_k}{(t^2+|x|^2)^{\frac{\theta}{2}}}(-\Delta)^{s} U_{t,p} \dd x}_{=:J_k} \\ & =-\int_{\Omega} \frac{z_k}{(t^2+|x|^2)^\frac{\theta}{2}}\mathscr{L}_{\mathcal{K}_{Y}} U_{t,p} \dd x.
	\end{split}
	\end{align}
	
Recalling the product rule for the fractional Laplacian, i.e. 
$$
(-\Delta)^s(u v)=u(-\Delta)^s v+v(-\Delta)^s u-\mathcal{I}_s(u, v),
$$
where
$$
\mathcal{I}_s(u, v)(x):=c_{N, s} \int_{\mathbb{R}^N} \frac{(u(x)-u(y))(v(x)-v(y))}{|x-y|^{N+2 s}} \dd y, \quad x \in \mathbb{R}^N,
$$ 
which holds for functions $u$ and $v$ such that $(-\Delta)^s u$ and $(-\Delta)^s v$ exist and
$$
\int_{\mathbb{R}^N} \frac{|(u(x)-u(y))(v(x)-v(y))|}{|x-y|^{N+2 s}}\dd y<\infty, 
$$
we compute 
\begin{align}
	I_{k}&:=\int_{\Omega} Y \cdot \nabla U_{t,p}(-\Delta)^{s}\left(\frac{z_k}{(t^2+|x|^2)^\frac{\theta}{2}}\right) \dd x \nonumber \\ &=\int_{\Omega} Y \cdot \nabla U_{t,p}\left[z_k(-\Delta)^{s}\left(\frac{1}{(t^2+|x|^2)^\frac{\theta}{2}}\right)+\frac{1}{(t^2+|x|^2)^\frac{\theta}{2}}(-\Delta)^{s} z_k-\mathcal{I}_s\left(z_k, \frac{1}{(t^2+|x|^2)^\frac{\theta}{2}}\right)\right] \dd x \nonumber\\
&:=I_{k}^{1}+I_{k}^{2}+I_{k}^{3}.\label{eq:33}
\end{align}
By continuity, we have 
	\begin{align}
		\lim _{k \rightarrow+\infty} I_{k}^{1}=\int_{\Omega} Y \cdot \nabla U_{t,p}(-\Delta)^{s}\left(\frac{1}{(t^2+|\cdot|^2)^\frac{\theta}{2}}\right) \dd x. \label{eq:34}
	\end{align}
		
On the other hand, by Lemma \ref{lm:dist}, we deduce	
\begin{align}\label{eq:35}
	\lim _{k \rightarrow+\infty} I_{k}^{2}=0=\lim _{k \rightarrow+\infty} I_{k}^{3}. 
\end{align}
To deal with $J_k$, we let $Y_t: \R^N\to \R^N, \, x\mapsto Y_t(x):= \frac{Y}{(t^2+|x|^2)^{\theta/2}}\in C^{0,1}(\R^N,\R^N)$. By \eqref{eq:p}, we have 
\begin{align*}
   J_k:&= \int_{\Omega} \frac{Y \cdot \nabla z_k}{(t^2+|x|^2)^\frac{\theta}{2}}(-\Delta)^{s} U_{t,p} \dd x=\int_{\Omega}Y_t\cdot\nabla z_k(-\Delta)^{s} U_{t,p}\dd x\\
   &=-\int_{\Omega}Y_t\cdot\nabla U_{t,p}(-\Delta)^{s} z_k  \dd x-\iint_{\R^N\times\R^N}(z_k(x)-z_k(y))(U_{t,p}(x)-U_{t,p}(y))\mathcal{K}_{Y_t}(x,y)\dd x\dd y,
\end{align*}
where
$\mathcal{K}_{Y_t}(\cdot,\cdot)$ is defined as in \eqref{eq:defor-kernel}. Since $\mathcal{K}_{Y_t}(\cdot,\cdot)$ is symmetric, we may write
$$
\iint_{\R^N\times\R^N}(z_k(x)-z_k(y))(U_{t,p}(x)-U_{t,p}(y))\mathcal{K}_{Y_t}(x,y)\dd xdy =\int_{\Omega}U_{t,p}\mathcal{L}_{\mathcal{K}_{Y_t}}z_k\dd x
$$
with
$$
\mathcal{L}_{\mathcal{K}_{Y_t}}(w)(x):=2\pv \int_{\R^N}(w(x)-w(y))\mathcal{K}_{Y_t}(x,y)dy.
$$
Putting everything together, we end up with
\begin{equation}\label{eq:3.6}
    J_k=-\int_{\Omega}Y_t\cdot\nabla U_{t,p}(-\Delta)^{s} z_k \dd x-\int_{\Omega}U_{t,p}\mathcal{L}_{\mathcal{K}_{Y_t}}z_k\dd x =: J_k^1 + J_k^2.
\end{equation}
From \eqref{eq:35}, we know that $J_k^1$ converges to zero as $k \to + \infty$. Next, arguing as in the proof of \eqref{eq:A.1} from Lemma \ref{lm:dist} and using the fact that 
$$
\big|\mathcal{K}_{Y_t}(x,y)\big|\leq C(t,\theta, N, s)|x-y|^{-N-2s}, \quad\text{for all $x,y\in\R^N, \,x\neq y$,}
$$
we deduce that $J_k^2$ converges to zero as well. In conclusion, we have that
\begin{align}\label{eq:3.7}
    \lim_{k\to\infty}J_k= \lim_{k\to\infty}\int_{\Omega} \frac{Y \cdot \nabla z_k}{(t^2+|x|^2)^\frac{\theta}{2}}(-\Delta)^{s} U_{t,p} \dd x=0.
\end{align}

Putting the computations in \eqref{eq:32}, \eqref{eq:33}, \eqref{eq:34}, \eqref{eq:35}, and \eqref{eq:3.7} together; using Lemma \ref{lm:compu} and passing to the limit (by the dominated convergence theorem) as $k \to + \infty$ and $t \to 0$,  we deduce 
\begin{align*}
	b_{N, s, \theta} \int_{\Omega} Y \cdot \nabla U_{t,p} \frac{\dd x}{|x|^{\theta+2 s}}-\theta \int_{\Omega} \frac{(-\Delta)^{s} U_{t,p}}{|x|^{\theta}} \dd x=-\int_{\Omega} \frac{\mathscr{L}_{\mathcal{K}_{Y}} U_{t,p} }{|x|^\theta}\dd x
\end{align*}
for all $u\in C^{1,1}_c(\Omega)$. Identity \eqref{eq:17} follows by integrating by parts in the first term of the identity above.
\end{proof}

Next, we proceed with the proof of Theorem \ref{th:main}.

\begin{proof}[Proof of Theorem \ref{th:main}]
For any $u\in C^{1,1}_c(\Omega)$ and $p\geq1$, let $U_{t,p}:=(u^2+t^2)^{p/2}-t^p\in C^{1,1}_c(\Omega)$ with $t>0$ be defined as above. Taking $Y=\textrm{id}_{\R^N}$ and recalling \eqref{eq:def-kernel}, identity \eqref{eq:17} becomes
\begin{align}\label{eq:sts}
	-b_{N, s, \theta}(N-2 s-\theta) \int_{\Omega} \frac{U_{t,p}}{|x|^{\theta+2 s}} \dd x-\theta \int_{\Omega} \frac{(-\Delta)^{s} U_{t,p}}{|x|^{\theta}} \dd x= -(N-2 s) \int_{\Omega} \frac{(-\Delta)^{s} U_{t,p}}{|x|^{\theta}} \dd x.
\end{align}
Since $U_{t,p}=\phi_t(u):=(t^2+u^2)^{p/2}-t^p$ and $\phi_t$ is convex, by Cordoba-Cordoba's inequality (see \cite[Theorem 1.1]{MR2032097}), we have 
	\begin{align*}
	b_{N, s, \theta} \int_{\Omega} \frac{U_{t,p}}{|x|^{\theta+2 s}} \dd x &=  \int_{\Omega} \frac{(-\Delta)^{s} U_{t,p}}{|x|^{\theta}} \dd x \\ 
	& \leq  \int_{\Omega} \frac{\phi'_t(u)(-\Delta)^{s} u}{|x|^{\theta}} \dd x =p\int_{\Omega}\frac{u}{(t^2+u^2)^{1/2}}\frac{(t^2+u^2)^{(p-1)/2}(-\Delta)^{s} u}{|x|^{\theta}} \dd x.
	\end{align*}
We distinguish two cases.

\textbf{Case 1}: $p=1$. Using  Fatou's lemma and then Lebesgue's dominated convergence theorem yields
\begin{align*}
	 b_{N, s, \theta} \int_{\Omega} \frac{|u|}{|x|^{\theta+2 s}} \dd x & \leq  b_{N, s, \theta} \liminf_{t\to 0^+}\int_{\Omega} \frac{U_{t,1}}{|x|^{\theta+2 s}} \dd x \\
  &\leq \lim_{t\to 0^+}\int_{\Omega} \frac{u}{(t^2+u^2)^{1/2}}\frac{(-\Delta)^{s} u}{|x|^{\theta}} \dd x \\ 
	& = \int_{\Omega} \frac{\mathrm{sign}(u)(-\Delta)^{s} u}{|x|^{\theta}} \dd x.
	\end{align*}

\textbf{Case 2:} $p>1$. Using Fatou's lemma, Lebesgue's dominated convergence theorem, and H\"older's inequality, we get
\begin{align*}
	b_{N, s, \theta} \int_{\Omega} \frac{|u|^p}{|x|^{\theta+2 s}} \dd x
	&  \leq p \int_{\Omega}  \frac{|u|^{p-1} }{|x|^\alpha} \, \frac{| (-\Delta)^s u| }{|x|^{\theta-\alpha}} \dd x\\
 & \le p\left(\int_{\Omega} \frac{|u|^p}{|x|^{\frac{\alpha p}{p-1}}} \dd x \right)^{\frac{p-1}{p}} \left(\int_{\Omega} \frac{|(-\Delta)^s u|^p}{|x|^{(\theta-\alpha)p}} \dd x\right)^{\frac{1}{p}}.
	\end{align*}
 The conclusion follows by letting $\alpha = (\theta + 2s){(p-1)}/{p}$.
\end{proof} 

Finally, we prove Corollary  \ref{cor:h}.

\begin{proof}[Proof of Corollary \ref{cor:h}]
We only need to prove that \eqref{eq:h-rem-m} holds for all $u\in \mathcal{H}^s_0(\Omega)$. Let $u\in \mathcal{H}^s_0(\Omega)$ and $u_n\in C^{1,1}_0(\Omega)$ such that 
$u_n\to u$ in $\mathcal{H}^s_0(\Omega)$. 
Applying H\"older's inequality and using the elementary estimate $(a+b)^2\leq 2(a^2+b^2)$ yields
\begin{align*}\label{eq2:cor2.2}
\int_{\Omega}\Big|\big|(-\Delta)^{s/2} u_n\big|^2-\big|(-\Delta)^{s/2} u\big|^2\Big|\dd x &\leq \int_{\Omega}\Big|(-\Delta)^{s/2} (u_n-u)\Big|\Big|(-\Delta)^{s/2} u_n+(-\Delta)^{s/2} u\Big|\dd x\nonumber\\
&\leq  \sqrt{2}\Big(\int_{\Omega}\Big|(-\Delta)^{s/2} (u_n-u)\Big|^2\dd x\Big)^{1/2}\Big([u_n]^2_{\mathcal{H}^s_0(\Omega)}+[u]^2_{\mathcal{H}^s_0(\Omega)}\Big)^{1/2}\nonumber\\
&\leq C\big[u_n-u\big]_{\mathcal{H}^s_0(\Omega)}\Big(1+[u]^2_{\mathcal{H}^s_0(\Omega)}\Big)^{1/2}\rightarrow 0\quad\text{as}\quad n\to\infty.
\end{align*}
We conclude by Fatou's lemma that
\begin{align*}
\left[\frac{b_{N, s / 2, s}}{2}\right]^{2}\int_{\Omega} \frac{u(x)^{2}}{|x|^{2 s}} \dd x \leq \left[\frac{b_{N, s / 2, s}}{2}\right]^{2}\liminf_{n\to\infty}\int_{\Omega} \frac{u_n(x)^{2}}{|x|^{2 s}} \dd x &\leq \lim_{n\to\infty}\int_\Omega |(-\Delta)^{s/2} u_n|^2 \dd x\\
&=\int_\Omega |(-\Delta)^{s/2} u|^2 \dd x.
\end{align*}
\end{proof}
\appendix

\section{Technical lemmas}
\label{app:prelim}

In this appendix, we collect the technical lemmas that have been used in the proof of the main results. 
\begin{lemma}\label{lm:dist}Let $\Omega$ be a bounded open set of class $C^{1,1}$. We define a cut-off function $\rho \in C^{\infty}(\R)$ with $\rho \equiv 1$ in $(-\infty,1]$ and  $\rho\equiv 0$ in $[2,+\infty)$ and consider $\rho_{k}(\cdot)=\rho\left(k \delta_{\Omega}(\cdot)\right)$, where $\delta_{\Omega}: \R^N\to \R$ is a $C^{1,1}(\R^N)$ function which coincides with the signed distance function near the boundary $\partial\Omega$ (note that, since the boundary is $C^{1,1}$, the sign distance function is $C^{1,1}$ as well). Moreover, we assume that $\delta_\Omega$ is positive in $\Omega$ and negative in $\R^N\setminus\Omega$. Let $z_k:=1-\rho_{k}$. Then, we have 
\begin{equation}
  z_k(t^2+|\cdot|^2)^{-\frac{\theta}{2}}\in C^{1,1}_c(\Omega) \qquad \text{and, hence, }\qquad\frac{z_k}{(t^2+|\cdot|^2)^{\frac{\theta}{2}}}\in \mathcal{H}^s_0(\Omega).
\end{equation}
Moreover, for any $u \in C^{1,1}_c(\Omega)$, 
\begin{align}\label{eq:A.1}
\int_{\Omega}Y\cdot\nabla u\frac{1}{(t^2+|x|^2)^{\frac{\theta}{2}}}\big[(-\Delta)^sz_k\big](x)\dd x &\to 0 && \text{as}\quad k\to\infty,
\\
\label{eq:A.2}
\int_{\Omega}Y\cdot \nabla u\, \mathcal{I}_s\left(z_k, \frac{1}{(t^2+|x|^2)^{\frac{\theta}{2}}}\right)\dd x&\to 0 && \text {as } \quad k \rightarrow +\infty,
\end{align}
where
$$
\mathcal{I}_s(v, w)(\cdot)=\frac{c_{N, s}}{2} \int_{\mathbb{R}^{N}} \frac{(v(\cdot)-v(\cdot+y))(w(\cdot)-w(\cdot+y))}{|y|^{N+2 s}}\dd y.
$$
\end{lemma}

\begin{proof} We start by proving \eqref{eq:A.1} and \eqref{eq:A.2}. Let $\Omega^{\prime\prime}\Subset \Omega^{\prime}\subset\Omega$ such that $\supp(u)\subset \Omega^{\prime\prime}$ and let $x\in \Omega^{\prime\prime}$. Since $\delta_\Omega(x)\geq c>0$, then, for $k$ sufficiently large, we have $\rho(k\delta_\Omega(x))=0$. Therefore, 
\begin{align*}
\big[(-\Delta)^{s}z_k\big](x)&= -c_{N, s} \pv\int_{\R^N}\frac{\rho_k(x)-\rho_k(y)}{|x-y|^{N+2s}}\dd y\\
&=c_{N,s}\int_{\Omega_{2/k}}\frac{\rho(k\delta_\Omega(y))}{|x-y|^{N+2s}}\dd y\\
&\leq \bar{c}(\Omega^{\prime},\Omega^{\prime\prime}, N,s)\,\textrm{vol}(\Omega_{2/k})
\end{align*}
for $k$ sufficiently large. It  follows that 
\begin{align*}
&\left|\int_{\Omega}Y\cdot\nabla u\frac{1}{(t^2+|x|^2)^{\frac{\theta}{2}}}\big[(-\Delta)^sz_k\big](x)\dd x\right|
\\ &\qquad \leq \bar{c}(\Omega^{\prime},\Omega^{\prime\prime})\textrm{vol}(\Omega_{2/k})\int_{\Omega}\left|Y\cdot\nabla u\frac{1}{(t^2+|x|^2)^\frac{\theta}{2}}\right| \dd x\rightarrow 0\quad\text{as}\quad k\to\infty,
\end{align*}
which gives \eqref{eq:A.1}. Similarly, for $x\in\Omega^{\prime\prime}$, we have 
\begin{align*}
    \left|\mathcal{I}_s\left(z_k, \frac{1}{(t+|x|^2)^{\frac{\theta}{2}}}\right)(x)\right|&=\frac{c_{N,s}}{2}\left|\int_{\Omega_{2/k}}\frac{-\rho(k\delta_\Omega(x))\big((t^2+|x|^2)^{-\frac{\theta}{2}}-(t^2+|y|^2)^{-\frac{\theta}{2}}\big)}{|x-y|^{N+2s}}dy\right|\\
    &\leq \bar{c}(\Omega^{\prime},\Omega^{\prime\prime}, N, s,t)\,\textrm{vol}(\Omega_{2/k}),
\end{align*}
for $k$ sufficiently large, from which \eqref{eq:A.2} follows.

The fact that $z_k(t^2+|\cdot|^2)^{-\frac{\theta}{2}}\in \mathcal{H}^s_0(\Omega)$ is standard but we present the proof here for completeness. It is clear that $\frac{z_k}{(t^2+|\cdot|^2)^{\frac{\theta}{2}}}\equiv 0$ in $\R^N\setminus\Omega$. Next, we take  $R_\Omega>1$ such that $\Omega\Subset B_{R_\Omega}$ and write 
\begin{align*}
    \frac{2}{c_{N,s}}\left[z_k(t^2+|\cdot|)^{-\frac{\theta}{2}}\right]^2_{H^s(\R^N)}&=\iint_{\R^N\times\R^N}\frac{\big(z_k(x)(t^2+|x|^2)^{-\frac{\theta}{2}}-z_k(y)(t^2+|y|^2)^{-\frac{\theta}{2}}\big)^2}{|x-y|^{N+2s}}\dd x\dd y\\
    &=\iint_{B_{R_\Omega}\times B_{R_\Omega}}\frac{\big(z_k(x)(t^2+|x|^2)^{-\frac{\theta}{2}}-z_k(y)(t^2+|y|^2)^{-\frac{\theta}{2}}\big)^2}{|x-y|^{N+2s}}\dd x\dd y\\
    &+\int_{\Omega}\frac{z_k^2(x)}{(t^2+|x|^2)^\theta}\int_{\R^N\setminus B_{R_\Omega}}\frac{\d y}{|x-y|^{N+2s}}\dd x\\
    &\leq 2\iint_{B_{R_\Omega}\times B_{R_\Omega}}\frac{\big(\rho_k(x)-\rho_k(y)\big)^2(t^2+|x|^2)^{-\theta}}{|x-y|^{N+2s}}\dd x\dd y\\
    & + 2\iint_{B_{R_\Omega}\times B_{R_\Omega}}\frac{\left((t^2+|x|^2)^{-\frac{\theta}{2}}-(t^2+|y|^2)^{-\frac{\theta}{2}}\right)^2}{|x-y|^{N+2s}}\dd x\dd y\\
    &+ c\int_{\Omega}\frac{\d x}{(t^2+|x|^2)^{\theta/2}}\int_{\R^N\setminus B_{R_\Omega}}\frac{\d y}{1+|y|^{N+2s}}\\
    &\leq \Bigg(\frac{2}{t^{2\theta}}\|\nabla \rho_k\|_{L^\infty(\R)}+2\Big[(t^2+|\cdot|^2)^{-\frac{\theta}{2}}\Big]_{C^1_{loc}(\R^N)}\Bigg)\iint_{B_{R_\Omega}\times B_{R_\Omega}}\frac{\d x\dd y}{|x-y|^{N+2s-2}}\\
    &+ c\int_{\Omega}\frac{\d x}{(t^2+|x|^2)^{\theta/2}}\int_{\R^N\setminus B_{R_\Omega}}\frac{\d y}{1+|y|^{N+2s}}\\
    &<\infty.
\end{align*}
\end{proof}

Finally, we shall give an explicit computation of the quantity $$
\lim _{t \rightarrow 0^{+}}(-\Delta)^{s}(t^2+|\cdot|^2)^{-\frac{\theta}{2}}(x).$$

\begin{lemma} \label{lm:compu}
Let $t>0$ and  $\theta>-2s$. Then, we have
\begin{align}\label{eq:26}
\lim _{t \rightarrow 0^{+}}(-\Delta)^{s}(t^2+|\cdot|^2)^{-\frac{\theta}{2}}(x)=|x|^{-(\theta+2 s)} b_{N, s, \theta},\quad\text{for all $x\in\R^N\setminus\{0\}$}.
\end{align}
Moreover, the convergence holds locally uniformly in $\R^N\setminus\{0\}$. Here, we introduced the following notation:
\begin{align} \label{eq:27}
b_{N, s, \theta}&=c_{N,s}\int_{0}^1r^{2s-1}(1-r^\theta)(1-r^{N-2s-\theta})\psi_N(r)\dd r, && {} \\
c_{N, s}&=\frac{s 4^{s} \Gamma\left(\frac{N+2 s}{2}\right)}{\pi^{\frac{N}{2}} \Gamma(1-s)}, && {}\\
 \psi_N(r)&=2\,\mathrm{vol}(\mathbb{S}^{N-2})\int_{-1}^1\frac{(1-h^2)^{\frac{N-3}{2}}}{\big(1+r^2-2rh\big)}\dd h &&  \text{if $N\geq 2$},\\
 \psi_N(r)&=2\left(\frac{1}{(1-r)^{1+2 s}}+\frac{1}{(1+r)^{1+2 s}}\right) && \text{if $N=1$}. 
\end{align} 
\end{lemma} 

\begin{proof}  
We only do the computation for $N\geq 2$. The case $N=1$ is similar, but  simpler. For the sake of brevity, we let $\psi:\R_+\to \R$ be defined by
\begin{align*}
\frac{1}{\mathrm{vol}(\mathbb{S}^{N-2})}\psi(r):&=\int_{0}^\pi\frac{\sin^{N-2}(\alpha_1)}{\big(1+r^2-2r\cos(\alpha_1)\big)^{(N+2s)/2}} \dd\alpha_1\\
&=2\int_{-1}^1\frac{(1-h^2)^{\frac{N-3}{2}}}{\big(1+r^2-2rh\big)^{(N+2s)/2}}\dd h.
\end{align*}
Let $x \in \mathbb{R}^{N}\setminus \{0\}$. By definition and passing into polar coordinates, we compute
\begin{align}\label{eq:A.9}
    &(-\Delta)^{s}\left[\frac{1}{(t^2+|\cdot|^ 2)^\frac{\theta}{2}}\right](x)= \pv \int_{\R^N}\Bigg(\frac{1}{(t^2+|x|^2)^{\frac{\theta}{2}}}-\frac{1}{(t^2+|y|^2)^{\frac{\theta}{2}}}\Bigg)\frac{\d y}{|x-y|^{N+2s}}\nonumber\\
 &=c_{N, s} \pv \int_{0}^{\infty}r^{N-1}\left((t^2+|x|^2)^{-\frac{\theta}{2}}-(t^2+r^2)^{-{\frac{\theta}{2}}}\right) \int_{\mathbb{S}^{N-1}} \frac{\d y}{\left(|x|^{2}+r^{2}-2 r|x| y_{1}\right)^{\frac{N+2 s}{2}}} \dd r \nonumber \\
&=c_{N, s} |x|^{-2s} \pv \int_{0}^{\infty} r^{N-1} \Bigg(\frac{1}{(t^2+|x|^2)^{\frac{\theta}{2}}}-\frac{1}{(t^2+r^2|x|^2)^{\frac{\theta}{2}}}\Bigg) \int_{\mathbb{S}^{N-1}} \frac{\d y}{\left(1+r^2-2ry_N\right)^{\frac{N+2 s}{2}}} \dd r\nonumber \\
&=2c_{N, s} |x|^{-2s} \pv \int_{0}^{\infty} r^{N-1} \Bigg(\frac{1}{(t^2+|x|^2)^{\frac{\theta}{2}}}-\frac{1}{(t^2+r^2|x|^2)^{\frac{\theta}{2}}}\Bigg) \nonumber\\ 
&\hspace{3cm} \times \mathrm{vol}(\mathbb{S}^{N-2})\int_{0}^\pi\frac{\sin^{N-2}(\alpha_1)}{\big(1+r^2-2r\cos(\alpha_1)\big)^{(N+2s)/2}} \dd\alpha_1 \nonumber\\
&=c_{N, s} |x|^{-2s}\pv \int_{0}^{\infty} r^{N-1} \Bigg(\frac{1}{(t^2+|x|^2)^{\frac{\theta}{2}}}-\frac{1}{(t^2+r^2|x|^2)^{\frac{\theta}{2}}}\Bigg) \psi(r)\dd r\nonumber\\
&=c_{N, s} |x|^{-2s} \, \bigg( \pv \int_{0}^{1}\cdots \dd r\ + \pv\int_{1}^\infty\cdots\dd r\bigg)\nonumber\\
&=c_{N, s} |x|^{-2s}\pv\int_{0}^1\psi(r)\Bigg(r^{N-1}\left(\frac{1}{(t^2+|x|^2)^{\frac{\theta}{2}}}-\frac{1}{(t^2+r^2|x|^2)^{\frac{\theta}{2}}}\right) \nonumber \\ &\hspace{5cm} +r^{2s-1}\left(\frac{1}{(t^2+|x|^2)^{\frac{\theta}{2}}}-\frac{1}{(t^2+\frac{|x|^2}{r^2})^{\frac{\theta}{2}}}\right)\Bigg)\dd r.
\end{align}
We claim that the integral above is finite. To prove this, we distinguish two cases.

\textbf{Case 1:} $r\in (0,\gamma]$ for some $\gamma\in (\frac{6-\sqrt{32}}{2},1)$.  We use the bound
\begin{align}\label{eq:A.10}
    &\Bigg| r^{N-1}\left(\frac{1}{(t^2+|x|^2)^{\frac{\theta}{2}}}-\frac{1}{(t^2+r^2|x|^2)^{\frac{\theta}{2}}}\right) +r^{2s-1}\left(\frac{1}{(t^2+|x|^2)^{\frac{\theta}{2}}}-\frac{1}{(t^2+\frac{|x|^2}{r^2})^{\frac{\theta}{2}}}\right)\Bigg|\nonumber\\
    &\leq \frac{r^{N-1}}{|x|^\theta}\left(1+\frac{1}{r^\theta}\right)+\frac{r^{2s-1}}{|x|^\theta}(1+r^\theta).
\end{align}

\textbf{Case 2:} $r\in [\gamma,1)$. We use the power series expansion
\begin{equation*}\label{eq:power-series}
    (1+x)^\alpha=\sum_{n=0}^\infty \frac{\alpha(\alpha-1)\cdots(\alpha-n+1)}{n!} \, x^n:=\sum_{n=0}^\infty c_{\alpha,n} \, x^n \qquad\text{(for  \,$|x|<1$)}
\end{equation*}
to get, for $t>0$ small enough,
\begin{align}\label{eq:A.11}
    &\Bigg|r^{N-1}\left(\frac{1}{(t^2+|x|^2)^{\frac{\theta}{2}}}-\frac{1}{(t^2+r^2|x|^2)^{\frac{\theta}{2}}}\right)+r^{2s-1}\left(\frac{1}{(t^2+|x|^2)^{\frac{\theta}{2}}}-\frac{1}{(t^2+\frac{|x|^2}{r^2})^{\frac{\theta}{2}}}\right)\Bigg|\nonumber\\
    &=\frac{1}{|x|^\theta}\Bigg|\sum_{k=0}^\infty c_{-\frac{\theta}{2},k}\big(1-r^{\theta+2k}\big)\big(r^{2s-1}-r^{N-1-\theta-2k}\big)\frac{t^{2k}}{|x|^{2k}}\Bigg|\nonumber\\
    &\leq \frac{1}{|x|^\theta}\Bigg((1-r^\theta)\big|r^{2s-1}-r^{N-1-\theta}\big|+\sum_{k=1}^\infty \big|c_{-\frac{\theta}{2},k}\big|\big|r^{2s-1+2k}-r^{N-1-\theta}\big|\frac{t^{2k}}{(r|x|)^{2k}}\Bigg)\nonumber\\
    &\leq  \frac{1}{|x|^\theta}\Bigg((1-r^\theta)\big|r^{2s-1}-r^{N-1-\theta}\big|+C\sum_{k=1}^\infty \frac{t^{2k}}{(r|x|)^{2k}}\Bigg)\nonumber\\
    & \leq  \frac{C}{|x|^\theta}(1-r^\theta)\big|r^{2s-1}-r^{N-1-\theta}\big|
\end{align}
for some $C=C(\theta)>0$. In the last line, we used the fact that the power series in the line before converges uniformly (in $r$) to zero and hence is controlled by $(1-r^\theta)|r^{2s-1}-r^{N-1-\theta}|$. 

Moreover, we have
\begin{align}\label{eq:A.12}
    \int_{0}^1 \Bigg(\chi_{(0,\gamma]}(r) &\bigg(r^{N-1}\left(1+\frac{1}{r^\theta}\right)+r^{2s-1}(1+r^\theta)\bigg)\psi(r)\nonumber\\
    & +\chi_{[\gamma,1)}(r)\frac{|1-r^\theta|}{(1-r)^{1+2s}}r^{2s-1}|1-r^{N-\theta-2s}|(1-r)^{1+2s}\psi(r)\Bigg)\dd r<\infty.
\end{align}
Indeed, since $\psi(r)$ is bounded near zero, it is easy to see that
\begin{equation}\label{eq:A.13}
\int_{0}^1 \chi_{(0,\gamma]}(r) \bigg(r^{N-1}\left(1+\frac{1}{r^\theta}\right)+r^{2s-1}(1+r^\theta)\bigg)\psi(r)\dd r<\infty
\end{equation}
provided that $\theta>-2s$.
On the other hand, since
$$
(1-r)^{1+2s}\psi(r)=\frac{4}{(2\sqrt{r})^{\frac{N-1}{2}}}\left(\frac{1-r}{2\sqrt{r}}\right)^{1+2s}\int_{0}^1\frac{(h(1-h))^{\frac{N-3}{2}}}{\Big((\frac{1-r}{2\sqrt{r}})^2+h\Big)^{\frac{N+2s}{2}}}\dd h,
$$
which is $C^{s+\kappa(s)}([\gamma,1])$ (for some  $\kappa$ depending on $s$) by the choice of $\gamma\in (\frac{6-\sqrt{32}}{2},1)$ and \cite[Lemma 2.1]{DjitteJarohs}, we have that 
\begin{equation}\label{eq:A.14}
\int_{0}^1\chi_{[\gamma,1)}(r)\frac{|1-r^\theta|}{(1-r)^{1+2s}}r^{2s-1}|1-r^{N-\theta-2s}|(1-r)^{1+2s}\psi(r)\dd r<\infty.
\end{equation}
The claim \eqref{eq:A.12} then follows from \eqref{eq:A.13} and \eqref{eq:A.14}. This implies, in view of \eqref{eq:A.10} and \eqref{eq:A.11}, that the integral in \eqref{eq:A.9} converges. Consequently,
\begin{align*}
    &(-\Delta)^{s}(t^2+|\cdot|^ 2)^{-\frac{\theta}{2}}(x)\\
&=\frac{c_{N,s}}{|x|^{2s}}\int_{0}^1\psi(r)\Bigg(r^{N-1}\left(\frac{1}{(t^2+|x|^2)^{\frac{\theta}{2}}}-\frac{1}{(t^2+r^2|x|^2)^{\frac{\theta}{2}}}\right) +r^{2s-1}\left(\frac{1}{(t^2+|x|^2)^{\frac{\theta}{2}}}-\frac{1}{(t^2+\frac{|x|^2}{r^2})^{\frac{\theta}{2}}}\right)\Bigg)\dd r.
\end{align*}
In view of \eqref{eq:A.10}, \eqref{eq:A.11}, and \eqref{eq:A.12}, by Lebesgue's  dominated convergence theorem, we have that
\begin{align}
    &\frac{|x|^{2s}}{c_{N,s}}\lim_{t\to 0^+}(-\Delta)^{s}(t^2+|\cdot|^ 2)^{-\frac{\theta}{2}}(x)\nonumber\\
    &=\lim_{t\to 0^+}\int_{0}^1\psi(r)\Bigg(\frac{r^{N-1}}{(t^2+|x|^2)^{\frac{\theta}{2}}}-\frac{r^{N-1}}{(t^2+r^2|x|^2)^{\frac{\theta}{2}}}+\frac{r^{2s-1}}{(t^2+|x|^2)^{\frac{\theta}{2}}}-\frac{r^{2s-1}}{(t^2+\frac{|x|^2}{r^2})^{\frac{\theta}{2}}}\Bigg)\dd r\nonumber\\
    &=\frac{1}{|x|^\theta}\int_{0}^1r^{2s-1}(1-r^\theta)(1-r^{N-2s-\theta})\psi(r)\dd r\nonumber.
\end{align}
In other words, we conclude that 
\begin{equation*}
    \lim_{t\to 0^+}(-\Delta)^{s}(t^2+|\cdot|^ 2)^{-\frac{\theta}{2}}(x)=|x|^{-(\theta+2s)}c_{N,s}\int_{0}^1r^{2s-1}(1-r^\theta)(1-r^{N-2s-\theta})\psi(r)\dd r.
\end{equation*}

\end{proof}

\vspace{3mm}

\section*{Acknowledgments}

We thank A. Nazarov and E. Zuazua for their encouragement. We also thank F. Glaudo and T. K\"onig  for several helpful conversations and M. M. Fall for valuable comments on the first draft of this paper.

Nicola De Nitti is a member of the Gruppo Nazionale per l’Analisi Matematica, la Probabilit\'a e le loro Applicazioni (GNAMPA) of the Istituto Nazionale di Alta Matematica (INdAM).

This work was partially supported by the Alexander von Humboldt-Professorship
program and by the Transregio 154 Project ``Mathematical Modelling, Simulation and Optimization Using the Example of Gas Networks'' of the Deutsche
Forschungsgemeinschaft.

\vspace{3mm}

\bibliographystyle{abbrv}
\bibliography{FracH-ref.bib}

\begin{thebibliography}{10}

\bibitem{MR3408787}
A.~A. Balinsky, W.~D. Evans, and R.~T. Lewis.
\newblock {\em The analysis and geometry of {H}ardy's inequality}.
\newblock Universitext. Springer, Cham, 2015.

\bibitem{MR2048514}
G.~Barbatis, S.~Filippas, and A.~Tertikas.
\newblock A unified approach to improved {$L^p$} {H}ardy inequalities with best
  constants.
\newblock {\em Trans. Amer. Math. Soc.}, 356(6):2169--2196, 2004.

\bibitem{2209.03012}
F.~Bianchi, L.~Brasco, and A.~C. Zagati.
\newblock On the sharp {H}ardy inequality in {S}obolev-{S}lobodeckiĭ spaces.
\newblock {\em ArXiv:2209.03012}, 2022.

\bibitem{MR3814363}
L.~Brasco and E.~Cinti.
\newblock On fractional {H}ardy inequalities in convex sets.
\newblock {\em Discrete Contin. Dyn. Syst.}, 38(8):4019--4040, 2018.

\bibitem{MR3642734}
L.~A. Caffarelli and Y.~Sire.
\newblock On some pointwise inequalities involving nonlocal operators.
\newblock In {\em Harmonic analysis, partial differential equations and
  applications}, Appl. Numer. Harmon. Anal., pages 1--18.
  Birkh\"{a}user/Springer, Cham, 2017.

\bibitem{MR2032097}
A.~C\'{o}rdoba and D.~C\'{o}rdoba.
\newblock A pointwise estimate for fractionary derivatives with applications to
  partial differential equations.
\newblock {\em Proc. Natl. Acad. Sci. USA}, 100(26):15316--15317, 2003.

\bibitem{zbMATH07213519}
A.~Cotsiolis and N.~Labropoulos.
\newblock On the {Hardy}-{Sobolev} inequalities.
\newblock In {\em Differential and integral inequalities}, pages 265--287.
  Cham: Springer, 2019.

\bibitem{MR2944369}
E.~Di~Nezza, G.~Palatucci, and E.~Valdinoci.
\newblock Hitchhiker's guide to the fractional {S}obolev spaces.
\newblock {\em Bull. Sci. Math.}, 136(5):521--573, 2012.

\bibitem{MR3865202}
J.~I. D\'{\i}az, D.~G\'{o}mez-Castro, and J.~L. V\'{a}zquez.
\newblock The fractional {S}chr\"{o}dinger equation with general nonnegative
  potentials. {T}he weighted space approach.
\newblock {\em Nonlinear Anal.}, 177(part A):325--360, 2018.

\bibitem{MR4288669}
D.~Dier, M.~Kassmann, and R.~Zacher.
\newblock Discrete versions of the {L}i-{Y}au gradient estimate.
\newblock {\em Ann. Sc. Norm. Super. Pisa Cl. Sci. (5)}, 22(2):691--744, 2021.

\bibitem{2112.10653}
S.~M. Djitte, M.~M. Fall, and T.~Weth.
\newblock A generalized fractional {P}ohozaev identity and applications.
\newblock {\em ArXiv:2112.10653}, 2021.

\bibitem{DjitteJarohs}
S.~M. Djitte and S.~Jarohs.
\newblock Nonradiality of second fractional eigenfunctions of thin annuli.
\newblock {\em Commun Pure Appl Math.}, 22(2):613--638, 2023.

\bibitem{MR3237044}
B.~Dyda and A.~V. V\"{a}h\"{a}kangas.
\newblock A framework for fractional {H}ardy inequalities.
\newblock {\em Ann. Acad. Sci. Fenn. Math.}, 39(2):675--689, 2014.

\bibitem{MR4480576}
D.~E. Edmunds and W.~D. Evans.
\newblock {\em Fractional {S}obolev spaces and inequalities}, volume 230 of
  {\em Cambridge Tracts in Mathematics}.
\newblock Cambridge University Press, Cambridge, 2023.

\bibitem{MR4062961}
M.~M. Fall.
\newblock Semilinear elliptic equations for the fractional {L}aplacian with
  {H}ardy potential.
\newblock {\em Nonlinear Anal.}, 193:111311, 29, 2020.

\bibitem{MR2425175}
R.~L. Frank, E.~H. Lieb, and R.~Seiringer.
\newblock Hardy-{L}ieb-{T}hirring inequalities for fractional {S}chr\"{o}dinger
  operators.
\newblock {\em J. Amer. Math. Soc.}, 21(4):925--950, 2008.

\bibitem{MR2469027}
R.~L. Frank and R.~Seiringer.
\newblock Non-linear ground state representations and sharp {H}ardy
  inequalities.
\newblock {\em J. Funct. Anal.}, 255(12):3407--3430, 2008.

\bibitem{MR2048513}
F.~Gazzola, H.-C. Grunau, and E.~Mitidieri.
\newblock Hardy inequalities with optimal constants and remainder terms.
\newblock {\em Trans. Amer. Math. Soc.}, 356(6):2149--2168, 2004.

\bibitem{MR775683}
P.~Grisvard.
\newblock {\em Elliptic problems in nonsmooth domains}, volume~24 of {\em
  Monographs and Studies in Mathematics}.
\newblock Pitman (Advanced Publishing Program), Boston, MA, 1985.

\bibitem{MR1544414}
G.~H. Hardy.
\newblock Note on a theorem of {H}ilbert.
\newblock {\em Math. Z.}, 6(3-4):314--317, 1920.

\bibitem{MR436854}
I.~W. Herbst.
\newblock Spectral theory of the operator
  {$(p\sp{2}+m\sp{2})\sp{1/2}-Ze\sp{2}/r$}.
\newblock {\em Comm. Math. Phys.}, 53(3):285--294, 1977.

\bibitem{MR2256532}
A.~Kufner, L.~Maligranda, and L.-E. Persson.
\newblock The prehistory of the {H}ardy inequality.
\newblock {\em Amer. Math. Monthly}, 113(8):715--732, 2006.

\bibitem{MR3888401}
M.~Kwa\'{s}nicki.
\newblock Fractional {L}aplace operator and its properties.
\newblock In {\em Handbook of fractional calculus with applications. {V}ol. 1},
  pages 159--193. De Gruyter, Berlin, 2019.

\bibitem{zbMATH02586070}
E.~Landau.
\newblock A note on a theorem concerning series of positive terms. {Extract}
  from a letter of {Prof}. {E}. {Landau} to {Prof}. {I}. {Schur} (communicated
  by {G}. {H}. {Hardy}).
\newblock {\em J. Lond. Math. Soc.}, 1:38--39, 1926.

\bibitem{MR2659764}
M.~Loss and C.~Sloane.
\newblock Hardy inequalities for fractional integrals on general domains.
\newblock {\em J. Funct. Anal.}, 259(6):1369--1379, 2010.

\bibitem{MR1431208}
T.~Matskewich and P.~E. Sobolevskii.
\newblock The best possible constant in generalized {H}ardy's inequality for
  convex domain in {${\bf R}^n$}.
\newblock {\em Nonlinear Anal.}, 28(9):1601--1610, 1997.

\bibitem{MR1986163}
V.~Mazya and T.~Shaposhnikova.
\newblock Erratum to: ``{O}n the {B}ourgain, {B}rezis and {M}ironescu theorem
  concerning limiting embeddings of fractional {S}obolev spaces'' [{J}.
  {F}unct. {A}nal. {\bf 195} (2002), no. 2, 230--238; {MR}1940355
  (2003j:46051)].
\newblock {\em J. Funct. Anal.}, 201(1):298--300, 2003.

\bibitem{MR1211727}
E.~Mitidieri.
\newblock A {R}ellich type identity and applications.
\newblock {\em Comm. Partial Differential Equations}, 18(1-2):125--151, 1993.

\bibitem{MR1769903}
E.~Mitidieri.
\newblock A simple approach to {H}ardy inequalities.
\newblock {\em Mat. Zametki}, 67(4):563--572, 2000.

\bibitem{zbMATH07319433}
R.~Musina and A.~I. Nazarov.
\newblock Complete classification and nondegeneracy of minimizers for the
  fractional {Hardy}-{Sobolev} inequality, and applications.
\newblock {\em J. Differ. Equations}, 280:292--314, 2021.

\bibitem{MR4160667}
R.~Musina and A.~I. Nazarov.
\newblock A note on higher order fractional {H}ardy-{S}obolev inequalities.
\newblock {\em Nonlinear Anal.}, 203:Paper No. 112168, 3, 2021.

\bibitem{MR1069756}
B.~Opic and A.~Kufner.
\newblock {\em Hardy-type inequalities}, volume 219 of {\em Pitman Research
  Notes in Mathematics Series}.
\newblock Longman Scientific \& Technical, Harlow, 1990.

\bibitem{MR0240668}
F.~Rellich.
\newblock {\em Perturbation theory of eigenvalue problems}.
\newblock Gordon and Breach Science Publishers, New York-London-Paris, 1969.
\newblock Assisted by J. Berkowitz, With a preface by Jacob T. Schwartz.

\bibitem{MR3211861}
X.~Ros-Oton and J.~Serra.
\newblock The {P}ohozaev identity for the fractional {L}aplacian.
\newblock {\em Arch. Ration. Mech. Anal.}, 213(2):587--628, 2014.

\bibitem{MR1717839}
D.~Yafaev.
\newblock Sharp constants in the {H}ardy-{R}ellich inequalities.
\newblock {\em J. Funct. Anal.}, 168(1):121--144, 1999.

\bibitem{zbMATH06446451}
J.~Yang.
\newblock Fractional {Sobolev}-{Hardy} inequality in {{\(\mathbb{R}^N\)}}.
\newblock {\em Nonlinear Anal., Theory Methods Appl., Ser. A, Theory Methods},
  119:179--185, 2015.

\end{thebibliography}

\vfill

\end{document}